\documentclass[11pt,thmsa]{article}

\usepackage{amssymb,latexsym}
\usepackage{amsmath,amscd}
\usepackage[all]{xy}

 \usepackage{mathrsfs}
 \usepackage{dsfont}
 \usepackage{amsmath}
 \usepackage{amsthm}
 \usepackage{amsfonts}
 \usepackage{amssymb}
 \usepackage{mathtools}
 \usepackage[pdftex]{graphicx}
 \usepackage{wrapfig}
 \usepackage{layout}
 \usepackage{verbatim}
 \usepackage{alltt}
 \usepackage{xypic}
 \usepackage{bbm}
 \usepackage{tikz}
 \input xy
 \xyoption{all}
\usepackage[shortlabels]{enumitem}

 \DeclareMathAlphabet{\mathpzc}{OT1}{pzc}{m}{it}

 \newtheorem{theorem}{Theorem}[section]
 \newtheorem{lemma}[theorem]{Lemma}
 \newtheorem{proposition}[theorem]{Proposition}

 \newtheorem{definition}[theorem]{Definition}

  \theoremstyle{definition}
 \newtheorem{example}[theorem]{Example}
 \newtheorem{remark}[theorem]{Remark}

\newtheorem*{theoremnn}{Theorem}
\newtheorem*{TheoremA}{Theorem A}

\theoremstyle{definition}
  
\newtheorem*{acknowledgements}{Acknowledgements}

\renewenvironment{proof}{\noindent{\it
Proof.}}{\bgroup\hspace{\stretch{1}}$\square$\egroup\medskip\par}

\newcommand{\Rep}{\mathsf{Rep}}

\newcommand{\id}{\mathrm{id}}
\newcommand{\End}{\mathsf{End}}

\newcommand{\Loc}{\operatorname{Loc}}
\newcommand{\hol}{\mathsf{hol}}

\newcommand{\A}{\mathsf{A}}

\newcommand{\fr}{\mathcal{F}(E,\partial)}
\newcommand{\ZZ}{\ensuremath{\mathbb{Z}}}
\newcommand{\RR}{\ensuremath{\mathbb{R}}}
\newcommand{\Rcal}{\mathcal{R}}
\newcommand{\gl}{\mathcal{GL}(E,\partial)}
\newcommand{\front}{P}
\newcommand{\back}{Q}
\begin{document}


\def\bleft{\draw(x) -- (y) -- (z)--(W);}
\def\bright{\draw(x) -- (y) -- (z)--(w);}

\vspace{15cm}
 \title{The higher Riemann-Hilbert correspondence and principal 2-bundles}
\author{Camilo Arias Abad\footnote{Universidad Nacional de Colombia, Medell\'in. email: camiloariasabad@gmail.com }\; and Sebasti\'an V\'elez V\'asquez\footnote{Universidad Nacional de Colombia, Medell\'in. email: svelezv@unal.edu.co } }

 \maketitle
 \begin{abstract}
In this note we show that the holonomies for higher local systems provided by the higher Riemann-Hilbert correspondence coincide with those associated to principal 2-bundles where both formalisms meet.
\end{abstract}
\tableofcontents

\section{Introduction}\label{intro}

The simplest version of the Riemann-Hilbert correspondence states that representations of the fundamental groupoid $\pi_1(M)$ of a manifold correspond naturally to flat vector bundles on $M$. This is an instance of Lie theory, the correspondence between representations of the groupoid $\pi_1(M)$, and those of its Lie algebroid $TM$. These structures are known as local systems, since they can be used as coefficients for cohomology. By fixing a base point $ \ast \in M$, one can identify the representations of $\pi_1(M)$ with those of
$\pi_1(M,\ast)\simeq \pi_0({\bf{\Omega}}(M))$, the group of connected components of  the based loop space of ${\bf \Omega}(M)$. There are natural identifications:
\[ \mathsf{Rep}(\pi_1(M, \ast))\simeq \mathsf{Rep}(\pi_0({\bf{\Omega}}(M)))\simeq H_0({\bf{\Omega}}(M))\text{-}\mathsf{Mod}.\]
Therefore, one can equivalently think of local systems as modules over the first homology of the based loop space.\\

The theory of higher local systems is what arises when one replaces the fundamental groupoid $\pi_1$ by the infinity groupoid $\pi_\infty$.
Each of the characterizations of local systems given above admit a natural {\it higher dimensional }generalization:
\begin{itemize}
\item[(A)] Flat vector bundles are replaced by flat superconnections.
\item[(B)] Representations of $\pi_1$ are replaced by representations of the infinity groupoid $\pi_\infty$.
\item[(C)] Modules over $H_0({\bf \Omega}(M))$ are replaced by modules over the algebra of singular chains on the based loop space $C( {\bf \Omega}(M))$.
\end{itemize}

Each of these notions can be organized as differential graded categories $\mathsf{Loc}_\infty(M)$, $\mathsf{\Rep}(\pi_\infty(M))$ and
$C( {\bf \Omega}(M))\text{-}\mathsf{Mod}$. It turns out that the correponding notions of higher local system are equivalent.
\begin{theoremnn}
The differential graded categories  $\mathsf{Loc}_\infty(M)$, $\mathsf{\Rep}(\pi_\infty(M))$ and
$C( {\bf \Omega}(M))\text{-}\mathsf{Mod}$ are $\mathsf{A}_\infty$-equivalent.
\end{theoremnn}

The theorem above is the conjunction of several works, including \cite{Block}, \cite{Rivera}, \cite{Holstein}, \cite{abadzgraded}, \cite{abadrham}, \cite{igusa}.
The equivalence between   $\mathsf{Loc}(X)$ and $\mathsf{\Rep}(\pi_\infty(M))$, proved by Block and Smith in \cite{Block}, is known as the higher Riemann-Hilbert correspondence. This theorem amounts to constructing higher dimensional versions of holonomies for superconnections. This problem turns out to be related to de Rham's theorem. The natural integration map:
\[ \psi: \Omega(M) \rightarrow C^{\bullet}(M), \,\,\,\, \varphi(\eta) (\sigma):= \int_{\Delta} \sigma^*(\eta),\]
induces an isomorphism of algebras in cohomology, even though it is not an algebra map. A more satisfying explanation of this situation
was provided by Gugenheim in \cite{Gugenheim}, who proved the following.

\begin{theoremnn}
The map $\psi$ can be naturally extended to an $\mathsf{A}_\infty$ equivalence $\overline{\psi}: \Omega(M) \rightarrow C^{\bullet}(M)$.
\end{theoremnn}

Gugenheim's $\mathsf{A}_\infty$ version of de Rham's theorem is proven by an explicit construction based on Chen's iterated integrals \cite{chen}.
The holonomies required by the higher Riemann Hilbert correspondence can be constructed by pushing forward Maurer-Cartan elements along
Gugenheim's $\mathsf{A}_\infty$ map. This ultimately provides explicit formulas for higher holonomies in terms of Chen's iterated integrals.\\

Ordinary gauge theory studies principal $G$-bundles, connections and holonomies. In higher gauge theory, groups are replaced by $2$-groups, principal bundles by principal $2$-bundles, and there are holonomies associated to two dimensional surfaces. There is a large literature on higher gauge theory, see for instance \cite{Schreiber1}, \cite{Schreiber2}, \cite{Schreiber3}, \cite{Baez}, \cite{Bartels}, \cite{waldorf1}, \cite{waldorf2}, \cite{faria11} and references therein. Both the formalism of higher local systems and that of principal 2-bundles deal with generalized, higher dimensional holonomies. We note that neither is more general than the other. On the one hand, the theory of higher local systems produces holonomies for simplices of all dimensions. On the other hand, the formalism of higher gauge theory allows for more general $2$-groups. \\

The purpose of this note is to compare these two formalisms where they meet. We show that higher local systems provide examples of principal 2-bundles and the corresponding constructions for holonomies coincide. Our main result, which is a global version of \cite{abadcompare}, is as follows.

\begin{TheoremA}
The map that associates a principal $2$-bundle with connection  to a higher local system is compatible with the construction of holonomies.  More precisely, the following diagram commutes:
  \begin{equation*}\label{comparison}
  \begin{CD}
    \Loc_\infty^{(E,\partial)}(M)@> K>> 2\text{-}\mathcal Bun^f_{\fr}(M)\\
    @V\mathcal{T}VV@VVTV\\
    \text{Rep}_{\leq2}(M,(E,\partial))@>> \mathcal{K}>\text{Rep}_{\leq2}(M,\fr)
  \end{CD}
\end{equation*}
\end{TheoremA}
\begin{acknowledgements}
We would like to acknowledge the support of Colciencias through  their grant {\it Estructuras lineales en topolog\'ia y geometr\'ia}, with contract number FP44842-013-2018.  We also thank  the Alexander von Humboldt foundation which supported our work through the Humboldt Institutspartnerschaftet { \it Representations of Gerbes and higher holonomies}. We thank Santiago Pineda for several conversations at early stages of this work, and Konrad Waldorf for his hospitality during our visits to Greifswald. This note was prepared for the proceedings of the meeting  {\it Geometry in Algebra and Algebra in Geometry,} held in Medell\'in in 2019. We would like to thank the organizers for the good time we had, and for the invitation to speak.
\end{acknowledgements}

\section{Principal 2-bundles and parallel transport}\label{c4s1}
In this subsection we review the formalism of connections on principal 2-bundles and the corresponding notions of connections and parallel transport. This presentation follows closely the work of Waldorf in \cite{waldorf1} and \cite{waldorf2}.

\begin{definition}
  A strict Lie 2-group is a small 2-category with a single object in which the set of 1-morphisms and the set of 2-morphisms are Lie groups, and all structure maps are Lie group homomorphisms.
\end{definition}

Strict 2-groups can be conveniently described in terms of crossed modules.

\begin{definition}
  A Lie Crossed Module $\Gamma$ is a tuple $\Gamma=(H, G, \tau,\alpha)$, where $H$, $G$ are Lie groups, $\tau: H\to G$ is a Lie group homomorphism and $\alpha: G\times H \to H$ is an action such that $\tau(\alpha(g,h))=g\tau(h) g^{-1}$ and $\alpha(\tau(h),h')=hh'h^{-1}$, for all $g\in G$ and $h, h'\in H$.
  \end{definition}

  There is natural correspondence between Lie crossed modules and strict 2-groups.
  Given a Lie crossed module $\Gamma=(H, G, \tau, \alpha)$, there is a Lie 2-group $\overline{\Gamma}$ which has $G$ as the space of 1-morphisms and $H \rtimes G$ as the space of  2-morphisms. The source and target maps on $\overline{\Gamma}$ are  given by $s(h,g)=g$ and $t(h,g)=\tau(h)g$, and the composition by the product on $G$. The correspondence $\Gamma \mapsto \overline{\Gamma}$ is bijective and we will abuse the notation and simply write $\Gamma$ for both the Lie crossed module and the corresponding 2-group.

\begin{definition}
  A differential Lie crossed module $\gamma$ is a tuple $\gamma=(\mathfrak{h}, \mathfrak{g}, t,a)$ where $\mathfrak{h}$ and $\mathfrak{g}$ are Lie algebras, $t:\mathfrak{h}\to\mathfrak{g}$ is a Lie algebra homomorphism and $a:\mathfrak{g} \rightarrow \mathsf{Der}(\mathfrak{h})$ is an action by derivations such that $a(t(Y_1)(Y_2))=[Y_1,Y_2]$ and $t((aX)(Y))=[X,t(Y)]$ for $X\in\mathfrak{g}$ and $Y_1,Y_2,Y\in\mathfrak{h}$.
\end{definition}
Clearly, by differentiating a  Lie crossed module one obtains a differential Lie crossed module.\\

Let us now briefly recall some definitions regarding Lie groupoids and anafunctors between them.
 A right action of a Lie groupoid $\mathcal X$ on a smooth manifold $M$ is an anchor map $\alpha: M\to\text{Obj}(\mathcal X)$ together with an action map $\circ:M\prescript{}{\alpha}{\times}_t\text{Mor}(\mathcal X)\to M$ that satisfy the following conditions
  \[
  (x\circ f)\circ g=x\circ(f\circ g),\quad x\circ\text{id}_{\alpha(x)}=x,\quad \alpha(x\circ g)=s(g),\quad x\in M,\; f,g\in\text{Mor}(\mathcal X).
  \]
  A left action is  a right action of the opposite Lie groupoid.
  A principal $\mathcal X$-bundle over a manifold $M$ is another manifold $P$ with a surjective submersion $\pi:P\to M$ and a right action of $\mathcal X$ on $P$ that preserves the projection and such that the map $P\prescript{}{\alpha}{\times}_t\text{Mor}(\mathcal X)\to P\times_M P$ given by $(p,f)\mapsto(p,p\circ f)$ is a diffeomorphism.
  Given Lie groupoids $\mathcal X,\;\mathcal Y$, an anafunctor $F:\mathcal X\to\mathcal Y$ is a smooth manifold $F$ called the total space, with a left action of $\mathcal X$ and a right action of $\mathcal Y$ such that the actions commute and the left anchor $\alpha_l:F\to\text{Obj}(\mathcal X)$ with the right action of $\mathcal Y$ is a principal $\mathcal Y$-bundle.
  Given anafunctors $F,F':\mathcal X\to\mathcal Y$, a transformation $f:F\Rightarrow F'$ is a smooth function $f:F\to F'$ that is $\mathcal X$-equivariant, $\mathcal Y$-equivariant and preserves the anchor maps.

\begin{remark}
  Any smooth functor $\mathcal F:\mathcal X\to\mathcal Y$ may be used to define an anafunctor $F:\mathcal X\to\mathcal Y$ with total space $F:=\text{Obj}(\mathcal X)\prescript{}{\mathcal F}{\times}_t\text{Mor}(\mathcal Y)$, anchors $\alpha_l(x,f):=x$ and $\alpha_r(x,f):=s(f)$ and actions $g\circ(x,f):=(t(g),\mathcal Fg\circ f)$ and $(x,f)\circ g:=(x,f\circ g)$. Thus any smooth functor may be considered an anafunctor. Similarly, a smooth natural transformation $\mathcal T:\mathcal F\Rightarrow \mathcal F'$ induces a transformation between the respective anafunctors given by $f_{\mathcal T}(x,g):=(x,\mathcal T(x)\circ g)$.
\end{remark}

\begin{definition}
  A \emph{right action} of a Lie 2-group $\Gamma$ on a groupoid $\mathcal X$ is a functor $R:\mathcal X \times \Gamma \to \mathcal X$ such that $R(p,1)=p$ and $R(\rho, \id_1)= \rho$, for all $p\in \text{Obj}(\mathcal X)$ and $\rho\in \text{Mor}(\mathcal X)$. Furthermore, if $m:\Gamma\times\Gamma\to \Gamma$ is the multiplication functor, then the following diagram must be commutative:
\[
\begin{CD}
\mathcal X\times\Gamma\times\Gamma @>\id_{\mathcal X}\times m>> \mathcal X\times\Gamma\\
@VR\times \id_{\Gamma}VV@VVRV\\
\mathcal X\times\Gamma @>>R> \mathcal X
\end{CD}
\]
If $\mathcal X$ and $\mathcal Y$ are groupoids with smooth right $\Gamma$-actions, a $\Gamma$-equivariant anafunctor $F:\mathcal X\to\mathcal Y$ is an anafunctor with a smooth action $\rho:F\times\text{Mor}(\Gamma)\to F$ that preserves the anchors, i.e. makes the following diagrams commutative
\[
\begin{CD}
  F\times\text{Mor}(\Gamma)@>\rho>>F\\
  @V\alpha_l\times tVV@VV\alpha_l V\\
 \text{Obj}(\mathcal X)\times\text{Obj}(\Gamma)@>>R>\text{Obj}(\mathcal X)
\end{CD}
\qquad\begin{CD}
  F\times\text{Mor}(\Gamma)@>\rho>>F\\
  @V\alpha_r\times sVV@VV\alpha_r V\\
 \text{Obj}(\mathcal Y)\times\text{Obj}(\Gamma)@>>R>\text{Obj}(\mathcal Y)
\end{CD}.
\]
Furthermore, $\rho$ must be compatible with the actions of $\mathcal X$ and $\mathcal Y$, which means that
\[
\rho(\chi\circ f\circ \eta,\gamma_l\circ\gamma\circ\gamma_r)=R(\chi,\gamma_l)\circ\rho(f,\gamma)\circ R(\eta,\gamma_r),
\]
for $\chi\in\text{Mor}(\mathcal X)$, $\eta\in\text{Mor}(\mathcal Y)$, $f\in F$ and $\gamma_l,\gamma,\gamma_r\in\text{Mor}(\Gamma)$.\\

If $F,F':\mathcal X\to\mathcal Y$ are $\Gamma$-equivariant anafunctors, then a transformation $f:F\Rightarrow F'$ is called $\Gamma$-equivariant if the smooth map $f:F\to F'$ is $\text{Mor}(\Gamma)$-equivariant.
\end{definition}
Given a smooth manifold $M$ we denote by $M_{dis}$ the trivial groupoid where the objects are the points of $M$ and the only morphisms are identities.
\begin{definition}
  A \emph{Principal $\Gamma$-2-Bundle} over $M$ is a groupoid $\mathcal P$ together with a right action $R:\mathcal P\times\Gamma\to\mathcal P$ and a functor $\pi:\mathcal P\to M_{dis}$ that is a surjective submersion at the level of objects. The action must preserve the functor $\pi$, and the functor
\[
(\text{pr}_1,R):\mathcal P\times\Gamma\to \mathcal P\times_M \mathcal P
\]
must be a weak equivalence.\\
\end{definition}
Principal $\Gamma$-2-bundles over a fixed manifold $M$ can be given the structure of a 2-category and, as such, can be classified by non-abelian cohomology. We will focus on an aspect of this classification which is how a principal $\Gamma$-2-bundle can be constructed from local data given by cocycles.
\begin{definition}
  Let $\Gamma=\{H,G,\tau,\alpha\}$ be a Lie crossed module and $\{U_i\}_{i\in I}$ an open covering of a manifold $M$. A $\Gamma$-cocycle on the covering is comprised of the following data:
  \begin{description}
    \item[1)] For each pair $i,j\in I$ a smooth map $g_{ij}:U_i\cap U_j\to G$.
    \item[2)] for each triple $i,j,k\in I$ a smooth map $a_{ijk}:U_i\cap U_j\cap U_k\to H$.
  \end{description}
  The cocycle conditions for the maps are the following:
  \begin{align*}
    g_{ik}&=(\tau\circ a_{ijk})g_{ij}g_{jk},&&\text{on }U_i\cap U_j\cap U_k\\
    a_{ijl}a_{jkl}&=a_{ikl}\alpha(g_{kl},a_{ijk}),&&\text{on }U_i\cap U_j\cap U_k\cap U_l
  \end{align*}
  We denote a $\Gamma$-cocycle by $(g,a)$.
\end{definition}
\begin{lemma}\label{cocycle1}
  Given a $\Gamma$-cocycle defined on an open covering of $M$,  there is an associated principal $\Gamma$-2-bundle over $M$.
\end{lemma}
\begin{proof}
  Let $(g,a)$ be a $\Gamma$-cocycle defined on the open covering $\{U_i\}_{i\in I}$. We define a Lie groupoid $\mathcal P$ as follows
  \[
  \text{Obj}(\mathcal P):=\bigsqcup_{i\in I}U_i\times G,\quad \text{Mor}(\mathcal P):=\bigsqcup_{i,j\in I}(U_i\cap U_j)\times H\times G.
  \]
  Source and target maps are
  \[
  s(i,j,x,h,g):=(j,x,g),\quad t(i,j,x,h,g):=(i,x,g_{ij}(x)^{-1}\tau(h)g).
  \]
  The composition of morphisms is given by
  \[(i,j,x,h_2,g_2)\circ(j,k,x,h_1,g_1):=(i,k,x,a_{ijk}(x)\alpha(g_{jk}(x),h_2)h_1,g_1).\]
  The projection $\pi:\mathcal P\to M_{dis}$ is $\pi(i,x,g):=x$. The action is defined by
  \[
  R((i,x,g),g'):=(i,x,gg'),\quad R((i,j,x,h,g),(h',g')):=(i,j,x,h\alpha(g,h'),gg').
  \]
  It is straight forward to check that $(\text{pr}_1,R):\mathcal P\times\Gamma\to \mathcal P\times_M \mathcal P$ is a weak equivalence.
\end{proof}
As in the classical setting, a connection on a principal $\Gamma$-2-bundle is defined in terms of a differential form. If $\Gamma$ is a Lie 2-group with differential Lie crossed module $\gamma$, then for any Lie groupoid $\mathcal P$ there is a differential graded Lie algebra $\Omega(\mathcal{P},\gamma)$. The details of the definition and properties of $\Omega(\mathcal{P},\gamma)$ can be found in \cite{waldorf1}. We will simply state that a differential form $\psi\in\Omega^k(\mathcal{P},\gamma)$ has three components $\psi=(\psi^a,\psi^b,\psi^c)$ with
\[
\psi^a\in\Omega^k(\text{Obj}(\mathcal P),\mathfrak g),\quad \psi^a\in\Omega^k(\text{Mor}(\mathcal P),\mathfrak h),\quad \psi^a\in\Omega^{k+1}(\text{Obj}(\mathcal P),\mathfrak h).
\]

\begin{example}
  Given a Lie 2-group $\Gamma=(H,G,\tau,\alpha)$ with differential Lie crossed module $\gamma=(\mathfrak{h}, \mathfrak{g}, t,a)$, there is a canonical Maurer-Cartan form $\Theta\in\Omega^1(\Gamma,\gamma)$. The form has two non-trivial components $\Theta^a\in\Omega^1(G,\mathfrak g)$ and $\Theta^b\in\Omega^1(H\times G,\mathfrak h)$ given by
  \[
  \Theta^a:=\theta^G\quad\text{and}\quad\Theta^b:=(\alpha_{\text{pr}^{-1}_G})_*(\text{pr}_H^*\theta^H),
  \]
  where $\theta^G$ and $\theta^H$ are the Maurer-Cartan forms of $G$ and $H$ respectively. The form $\Theta$ satisfies the Maurer-Cartan equation.
\end{example}
\begin{definition}
  Let $F:\mathcal P\to\Gamma$ be a smooth functor. Write $F_0:\text{Obj}(\mathcal P)\to G$ for the map at the level of objects, $F_G:\text{Mor}(\mathcal P)\to G$ and $F_H:\text{Mor}(\mathcal P)\to H$ for the projections onto $G$ and $H$ respectively of the map at the level of morphisms. There is a linear map $\text{Ad}_F:\Omega^k(\mathcal P,\gamma)\to \Omega^k(\mathcal P,\gamma)$ called the adjoint action defined as follows:
  \begin{align*}
    \text{Ad}_F(\psi)^a & :=\text{Ad}_{F_0}(\psi^a) \\
    \text{Ad}_F(\psi)^b & :=\text{Ad}_{F_H}((\alpha_{F_G})_*(\psi^b)+ (\tilde\alpha_{F_H^{-1}})_*(\text{Ad}_{F_G}(s^*\psi^a))\\
    \text{Ad}_F(\psi)^c & :=(\alpha_{F_0})_*(\psi^c)
  \end{align*}
\end{definition}
\begin{definition}
If $\mathcal{P}$ is a principal $\Gamma$-2-bundle, then a connection on $\mathcal{P}$ is a $1$-form $\Omega\in\Omega^1(\mathcal{P},\gamma)$ such that the following equation holds over $\mathcal{P}\times \Gamma$:

\[R^*\Omega=\text{Ad}^{-1}_{\text{pr}_\Gamma}(\text{pr}^*_{\mathcal P}\Omega)+\text{pr}^*_\Gamma\Theta.\]

 Unraveling this definition we find that a connection is a triple of ordinary differential forms $\Omega=(\Omega^a,\Omega^b,\Omega^c)$ with $\Omega^a\in\Omega^1(\mathcal P_0,\mathfrak g)$, $\Omega^b\in\Omega^1(\mathcal P_1,\mathfrak h)$ and $\Omega^c\in\Omega^2(\mathcal P_0,\mathfrak h)$ such that
$$\begin{array}{lc}
R^{*}\Omega^a  =  \text{Ad}_g^{-1}(p^{*}\Omega^a)+g^{*}\Theta & \text{over} \ \text{Obj}(\mathcal P) \times \text{Obj}(\Gamma) \\
R^{*}\Omega^b  =  (\alpha_{g^{-1}})_{*}(\text{Ad}_h^{-1}(p^{*}\Omega^b)+(\tilde{\alpha}_h)_{*}(p^{*}s^{*}\Omega^a)+h^{*}\Theta) &  \text{over} \ \text{Mor}(\mathcal P) \times \text{Mor}(\Gamma)\\
R^{*}\Omega^c  =  (\alpha_{g^{-1}})_{*}(p^{*}\Omega^c) & \text{over} \  \text{Obj}(\mathcal P) \times \text{Obj}(\Gamma).
\end{array}$$
Here $p, g$ and $h$ are the projections to either $\text{Obj}(\mathcal P)$ or $\text{Mor}(\mathcal P)$, $G$ and $H$, respectively. \\

The curvature of a connection is $\text{curv}(\Omega)=D\Omega+\frac{1}{2}[\Omega\wedge\Omega]\in\Omega^2(\mathcal P,\gamma)$. The connection is called flat if $\text{curv}(\Omega)=0$.
\end{definition}

Principal $\Gamma$-2-bundles over $M$ with flat connection fit into a 2-category which we denote $2\text{-}\mathcal{B}un_\Gamma^{f}(M)$. The details of the bicategory structure can be found in \cite{waldorf1}. Furthermore, if we fix a principal $\Gamma$-2-bundle $\mathcal P$, the category of flat connections defined over $\mathcal P$ will be denoted $2\text{-}\mathcal{B}un_{\mathcal P}^f(M)$. The classification of principal $\Gamma$-2-bundles over $M$ by non-abelian cohomology may be refined to classify bundles with flat connection. First we define a differential $\Gamma$-cocycle, which is a $\Gamma$-cocycle with the local data required to build a connection.
\begin{definition}
  Let $\Gamma=(H,G,\tau,\alpha)$ be a Lie crossed module and $\{U_i\}_{i\in I}$ an open covering of a manifold $M$. A differential $\Gamma$-cocycle on the covering is comprised of the following data:
  \begin{description}
    \item[1)] A $\Gamma$-cocycle $(g,a)$.
    \item[2)] On every open set $U_i$ a couple of forms $(A_i,B_i)$ where $A_i\in\Omega^1(U_i,\mathfrak g)$ and $B_i\in\Omega^2(U_i,\mathfrak h)$.
    \item[3)] For each pair $i,j\in I$, a form $\varphi_{ij}\in\Omega^1(U_i\cap U_j,\mathfrak h)$.
  \end{description}
  The forms must satisfy the following conditions:
  \begin{align}
    A_j+\tau_*(\varphi_{ij}) & =\text{Ad}_{g_{ij}}(A_i)-g_{ij}^*\bar\theta &&\text{on }U_i\cap U_j\label{diffcocycle1}\\
    (\alpha_{g_{ij}})_*(B_i) & =B_j+d\varphi_{ij}+\frac{1}{2}[\varphi_{ij}\wedge \varphi_{ij}]+\alpha_*(A_j\wedge\varphi_{ij})&&\text{on }U_i\cap U_j\label{diffcocycle2}\\
    \varphi_{jk}+(\alpha_{g_{jk}})_*(\varphi_{ij})-a_{ijk}^*\theta&=\text{Ad}^{-1}_{a_{ijk}}(\varphi_{ik})+(\tilde\alpha_{a_{ijk}})_*(A_k)&&\text{on }U_i\cap U_j\cap U_k\label{diffcocycle3}
  \end{align}
  A pair $(A_i,B_i)$ satisfying the conditions stated above is called a $\Gamma$-connection. We denote differential $\Gamma$-cocycles by $(g,a,A,B,\varphi)$
  A differential $\Gamma$-cocycle has two curvatures associated to it. The ``3-curvature" is
  \begin{equation}\label{3curv}
  \text{curv}(A_i,B_i):=dB_i+\alpha_*(A_i)( B_i)\in\Omega^3(U_i,\mathfrak h).
  \end{equation}
  The ``fake curvature" is
  \begin{equation}\label{fcurv}
  \text{fcurv}(A_i,B_i):=dA_i+\frac{1}{2}[A_i, A_i]-\tau_*(B_i)\in\Omega^2(U_i,\mathfrak g).
  \end{equation}
  A $\Gamma$-cocycle is called flat if both curvatures vanish.
\end{definition}
A flat differential $\Gamma$-cocycle encodes the necessary local data to construct a principal $\Gamma$-2-bundle with flat connection, that is the content of the following lemma:
\begin{lemma}\label{cocycle2}
  Given a differential $\Gamma$-cocycle defined on an open covering of $M$, there is an associated principal $\Gamma$-2-bundle with a flat connection over $M$.
\end{lemma}
\begin{proof}
  Let $\mathcal P$ be the bundle obtained from lemma \ref{cocycle1}. We define a connection on $\mathcal P$ as follows:
  \begin{align*}
    \Omega^a|_{U_i\times G} & := \text{Ad}_{\text{pr}_G}^{-1}(A_i)+\text{pr}_G^*\theta\\
    \Omega^b|_{(U_i\cap U_j)\times H\times G} & :=(\alpha_{\text{pr}_G^{-1}})_*(\text{Ad}^{-1}_{\text{pr}_H}(\varphi_{ij})+(\tilde\alpha_{\text{pr}_H})_*(A_j)+\text{pr}_H^*\theta)\\
    \Omega^c|_{U_i\times G}&:=-(\alpha_{\text{pr}_G^{-1}})_*(B_i).
  \end{align*}
  The forms defined above meet the criteria required of a flat connection on $\mathcal P$.
\end{proof}
A flat connection on a principal 2-bundle over $M$ can be used to define holonomies along paths and surfaces on $M$. The holonomies and their properties are better understood as a representation of the fundamental 2-groupoid of $M$.  Let us fix our conventions regarding simplices.
\begin{definition}
  The simplex of dimension $k$ for $k>0$ is the set
  \[
  \Delta_k=\{(t_1,\cdots,t_k)\in\RR^k\;\big|\; 1\geq t_1\geq\cdots\geq t_k\geq0\}.
  \]
  The zero-simplex $\Delta_0$ is the single-point space.
\end{definition}
The $k$-simplex $\Delta_k$ has $k+1$ vertices which we denote $v_0,\cdots, v_k$. It is straight forward to check that $v_i=(1,\cdots,1,0,\cdots,0)$ where the number of ones is equal to $i$. The edges of the simplex are the line segments connecting two different vertices, they will be denoted $[v_i,v_j]$ for $i\neq j$.
\begin{definition}
  The fundamental 2-groupoid of $M$, denoted $\pi_{\leq2}(M)$, is a bicategory comprised of the sets $X_k=\text{Hom}_{\text{PS}}(\Delta_k,M)$ for $k\leq2$ where $\text{PS}$ stands for piecewise smooth. The bicategory structure is as follows:
  \begin{itemize}
    \item The objects of $\pi_{\leq2}(M)$ are the points of $M$, so $\text{Obj}(\pi_{\leq2}(M))=\text{Hom}_{\text{PS}}(\Delta_0,M)$.
    \item A 1-morphism between the points $x,y\in M$ is a path connecting $x$ to $y$, therefore $1-\text{Mor}(\pi_{\leq2}(M))=\text{Hom}_{\text{PS}}(\Delta_1,M)$. The composition of morphisms is the usual composition of paths.
    \item The 2-morphisms of $\pi_{\leq2}(M)$ are fixed ends homotopies between paths modulo higher homotopies. A map $\sigma:\Delta_2\to M$ is a fixed ends homotopy between the paths $\sigma([v_0,v_2])$ and $\sigma([v_0,v_1]\ast[v_1,v_2])$, thus $2-\text{Mor}(\pi_{\leq2}(M))=\text{Hom}_{\text{PS}}(\Delta_2,M)$ with the usual vertical and horizontal compositions of homotopies.
  \end{itemize}
\end{definition}
\begin{remark}\label{Theta2}
  To elaborate on how a mapping $\sigma:\Delta_2\to M$ may be considered a homotopy, consider the maps $q:I^2\to\Delta_2$ defined by $q(t,s)=(\text{max}\{t,s\},s)$, and $\lambda:I^2\to I^2$ determined by the property that $\lambda(-,s)$ is the piecewise linear path running through the points $(s,1)\to(s,0)\to(0,0)$ and arrives at $(s,0)$ at time $t=1/2$. Then we denote the composition $q\circ\lambda$ by $\Theta_2$. The map $\Theta_2$ sends horizontal segments on $I^2$ into piecewise linear paths on $\Delta_2$.\\

\begin{figure}[h!]\label{theta2}
  \begin{center}
  \includegraphics[width=\textwidth]{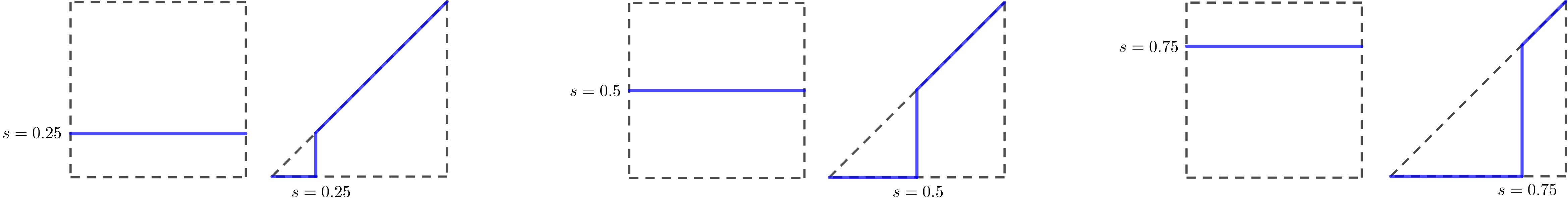}
  \end{center}
\end{figure}
Therefore, for a 2-simplex $\sigma:\Delta_2\to M$, the composition $\sigma\circ\Theta_2:I^2\to M$ is a homotopy with fixed ends between the paths $\sigma([v_0,v_2])$ and $\sigma([v_0,v_1]\ast[v_1,v_2])$.
\end{remark}
\begin{definition}
  Let $\Gamma$ be a Lie 2-group and $\mathcal P$ a principal $\Gamma$-2-bundle over $M$. The $\Gamma$-torsor category of $\mathcal P$ is the 2-category $\Gamma\text{-Tor}(\mathcal P)$ such that
  \begin{itemize}
    \item The objects of $\Gamma\text{-Tor}(\mathcal P)$ are the fibers $\mathcal P_x$ for $x\in M$. These are called $\Gamma$-torsors.
    \item The 1-morphisms are $\Gamma$-equivariant anafunctors $\mathcal P_x\Rightarrow\mathcal P_y$.
    \item The 2-morphisms are $\Gamma$-equivariant transformations between anafunctors.
  \end{itemize}
  A 2-group representation of $\pi_{\leq2}(M)$ on $\mathcal P$ is a 2-functor $\mathcal S:\pi_{\leq2}(M)\to\Gamma\text{-Tor}(\mathcal P)$. The representations  of $\pi_{\leq2}(M)$ on $\mathcal P$ will be denoted $\text{Rep}_{\leq2}(M,\mathcal P)$.
\end{definition}
We will now summarize the construction, which appeared in \cite{waldorf2}, that assigns a representation of the fundamental 2-groupoid to a flat connection on a principal 2-bundle over $M$. We begin by defining the 2-group representation locally.\\

From lemma \ref{cocycle2} we know that a principal $\Gamma$-2-bundle with flat connection is locally determined by a differential $\Gamma$-cocycle. So let us consider a 2 group $\Gamma=(H,G,\tau,\alpha)$ with differential Lie crossed module $\gamma=(\mathfrak h,\mathfrak g,\tau_*,\alpha_*)$, and the principal $\Gamma$-2-bundle $\mathcal P=M_{dis}\times \Gamma$. Let $(A,B)$ be a flat $\Gamma$-connection over $M$, this is $A\in\Omega^1(M,\mathfrak g)$, $B\in\Omega^2(M,\mathfrak h)$ are forms such that the curvatures (\ref{3curv}) and (\ref{fcurv}) vanish. A representation of $\pi_{\leq2}(M)$ in $\Gamma$ is defined as follows
\begin{itemize}
	\item Let $\gamma:\Delta_1\to M$ be a piecewise smooth path. Let $g_\gamma:\Delta_1\to G$ be the path that satisfies the initial condition $g_{\gamma}(0)=\text{id}_G$ and the differential equation
\begin{equation}\label{poe}
\dfrac{dg_{\gamma}(t)}{dt}=-\left(R_{g_{\gamma}(t)}\right)_{*} A\left(\dfrac{d\gamma}{dt}\right)
\end{equation}
where $(R_g)_{*}$ is the differential at $\text{id}_G\in G$ of the map given by right multiplication with $g$.\\
We define the smooth functor $\text{Hol}_\gamma:\mathcal P_{\gamma(0)}\Rightarrow\mathcal P_{\gamma(1)}$ at the level of objects by $\text{Hol}_\gamma(\gamma(0),g)=(\gamma(1),g_{\gamma}(1)g)$, which is clearly a $G$-equivariant map. If $(\text{id}_{\gamma(0)},h,g):(\gamma(0),g)\to (\gamma(0),\tau(h)g)$ is a morphism, we define $\text{Hol}_\gamma(\text{id}_{\gamma(0)},h,g)=(\gamma(1),\alpha(g_{\gamma}(1),h),g_\gamma(1)g).$
This is indeed a morphism $g_{\gamma}(1)g\to g_{\gamma}(1)\tau(h)g$ since
\[
\tau(\alpha(g_{\gamma}(1),h))g_\gamma(1)g=g_{\gamma}(1)\tau(h)g_{\gamma}(1)^{-1}g_{\gamma}(1)g=g_{\gamma}(1)\tau(h)g.
\]
Thus $\text{Hol}_\gamma$ is a $\Gamma$-equivariant functor.
	\item Let $\sigma:\Delta_2\to M$ be a 2-simplex and $\Sigma:=\sigma\circ\Theta_2:I^2\to M$. Denote by $\gamma_0$ and $\gamma_1$ the paths corresponding to $\sigma([v_0,v_1]\ast[v_1,v_2])$ and $\sigma([v_0,v_2])$ respectively. We write $\Sigma_s(t)=\Sigma(s,t)$, note that $\Sigma_0=\gamma_0$ and $\Sigma_1=\gamma_1$. The holonomy $\text{Hol}_{\sigma}$ is a $\Gamma$-equivariant transformation $\text{Hol}_{\sigma}:\text{Hol}_{\gamma_0}\Rightarrow\text{Hol}_{\gamma_1}$ defined as follows: let $h_\sigma:[0,1]\to H$ be the path that satisfies the initial condition $h_\sigma(0)=\text{id}_H$ and the differential equation
\begin{equation}\label{soe}
\dfrac{dh(s)}{ds}=\left(L_{h(s)}\right)_{*}\left(\int_0 ^1 \alpha\left(g_{\Sigma_s(t)^{-1}}\right)_{*}\left(B\left(\dfrac{\partial}{\partial t} \Sigma_s(t), \dfrac{\partial}{\partial s} \Sigma_s(t) \right)\right)dt\right),
\end{equation}
where $\left(L_{h}\right)_{*}$ is the differential at $\text{id}_H\in H$ of the map given by left multiplication with $h$. The transformation is defined by $\text{Hol}_\sigma(\sigma(v_0),g)=(\sigma(v_2),h_\sigma(1),g_{\gamma_0}g):(\sigma(v_2),g_{\gamma_0}g)\to(\sigma(v_2),g_{\gamma_1}g)$.
\end{itemize}
 For the global case, we consider a principal $\Gamma$-$2$-bundle $\mathcal{P}$ over $M$ with flat connection $\Omega$. For each path $\gamma:[0,1]\to M$ and each 2-simplex $\sigma:\gamma\Rightarrow \gamma'$ in $M$, we define an anafunctor $\text{Hol}_{\gamma}:\mathcal{P}_{\gamma(0)}\to \mathcal{P}_{\gamma(1)}$ and a $\Gamma$-equivariant transformation $\text{Hol}_{\sigma}:\text{Hol}_{\gamma}\Rightarrow \text{Hol}_{\gamma'}$ respectively.\\

For paths $\gamma:[0,1]\to M$ the idea is to define the set $\text{Hol}_{\gamma}(\lambda)$ with respect to a fixed subdivision $\lambda$ of $[0,1]$, then a smooth manifold $\text{Hol}_{\gamma}$ is constructed by taking a direct limit over the possible subdivisions $\lambda$.

\begin{definition}
Let $\pi:\mathcal{P}\to M_{\text{dis}}$ be a principal $\Gamma$-bundle with a connection $\Omega=(\Omega^a,\Omega^b,\Omega^c)$.
\begin{itemize}
	\item A path $\beta:[a,b]\to \text{Obj}(\mathcal{P})$ is \emph{horizontal}, if $\Omega^a({\beta'}(t))=0$ for all $t\in [a,b]$.
	\item A path $\rho:[a,b]\to \text{Mor}(\mathcal{P})$ is \emph{horizontal}, if $\Omega^b({\rho'}(t))=0$ for all $t\in [a,b]$.
\end{itemize}
\end{definition}

\begin{definition}
For $0<n\in \mathbb{N}$, let $T_n:=\{(t_i)_{i=0}^n \ | \ 0=t_0< t_1 < \cdots <t_n =1\}$ be the set of possible $n$-fold subdivisions of the interval $[0,1]$. For $\lambda\in T_n$ we define the set
\begin{align*}
\text{Hol}_{\gamma}(\lambda):=\{(\{\rho_i\}_{i=0}^n, \{\gamma_i\}_{i=1}^n) \;\big|\; &\rho_i\in \text{Mor}(\mathcal P),\; \gamma_i:[t_{i-1},t_i]\to \text{Obj}(\mathcal P)\; \text{are horizontal paths,}\\
  &\pi\circ\gamma_i=\gamma|_{[t_{i-1}, t_i]}, \;  t(\rho_i)=\gamma_{i+1}(t_i) \; \text{and} \; s(\rho_i)=\gamma_i(t_i)\}/ \sim.
\end{align*}
\end{definition}
We think about the elements of $\text{Hol}_{\gamma}(\lambda)$ as formal compositions of paths in $\text{Obj}(\mathcal P)$ and morphisms in $\text{Mor}(\mathcal P)$, using the notation $\xi=\rho_0\ast \gamma_1\ast \rho_1 \ast\cdots\ast \gamma_n\ast \rho_n$ for a representative $\xi$ of an element in $\text{Hol}_{\gamma}(\lambda)$. The relation $\sim$ is the relation generated by $\{\sim_j\}_{1\leq j\leq n}$ where
\[
\rho_0\ast \gamma_1\ast \cdots\ast \gamma_n\ast \rho_n\sim_j \rho_0'\ast \gamma_1' \ast\cdots\ast \gamma_n'\ast \rho_n'
\]
if there exist a horizontal path $\tilde{\rho}:[t_{j-1},t_j]\to \text{Mor}(\mathcal P)$ such that:
\begin{itemize}
	\item $\gamma_j=s(\tilde{\rho})$, $\gamma'_j=t(\tilde{\rho})$ and $\gamma'_i=\gamma_i$ for all $1\leq i \leq n$, $i\neq j$,
	\item $\rho'_{j-1}=\tilde{\rho}(t_{j-1})\circ \rho_{j-1}$, $\rho'_{j}=\rho_{j}\circ \tilde{\rho}(t_{j})^{-1}$ and $\rho'_i=\rho_i$ for all $0\leq i\leq n, i\neq j, j-1$.
\end{itemize}

The anchor maps are given by:
\[\alpha_l:\text{Hol}_{\alpha}(\lambda)\to \mathcal{P}_{\gamma(0)} \ : \ \rho_0\ast \gamma_1\ast \cdots\ast \gamma_n\ast \rho_n\mapsto s(\rho_0),\]
\[\alpha_r:\text{Hol}_{\alpha}(\lambda)\to \mathcal{P}_{\gamma(1)} \ : \ \rho_0\ast \gamma_1\ast \cdots\ast \gamma_n\ast \rho_n\mapsto t(\rho_n).\]

The left $\mathcal{P}_{\gamma(0)}$-action $\text{Mor}(\mathcal{P}_{\gamma(0)})_{ \ s}\times_{\alpha_{l}}\text{Hol}_{\gamma}(\lambda)\to \text{Hol}_{\gamma}(\lambda)$ and right $\mathcal{P}_{\gamma(1)}$-action $\text{Hol}_{\gamma}(\lambda)_{ \ \alpha_{r}}\times_{t} \ \text{Mor}(\mathcal{P}_{\gamma(1)})\to \text{Hol}_{\gamma}(\lambda)$ on the set $\text{Hol}_{\gamma}(\lambda)$ are defined by
\[
\rho\circ(\rho_0\ast \gamma_1\ast \cdots\ast \gamma_n\ast \rho_n):=(\rho_0\circ\rho^{-1})\ast \gamma_1\ast \cdots\ast \gamma_n\ast \rho_n,
\]
\[
(\rho_0\ast \gamma_1\ast \cdots\ast \gamma_n\ast \rho_n)\circ\rho:=\rho_0\ast \gamma_1\ast \cdots\ast \gamma_n\ast (\rho^{-1}\circ \rho_n).
\]

The $\text{Mor}(\Gamma)$-action $ \text{Hol}_{\gamma}(\lambda)\times \text{Mor}(\Gamma)\to  \text{Hol}_{\gamma}(\lambda)$ is given by
\[
(\rho_0\ast \gamma_1\ast \cdots\ast \gamma_n\ast \rho_n)\cdot(h,g):= R(\rho_0,(h^{-1},\tau(h)g))\ast R(\gamma_1,g)\ast\cdots\ast R(\gamma_n,g)\ast R(\rho_n,g).
\]

Finally, to define the smooth manifold $\text{Hol}_{\gamma}$ consider the set $T:=\bigsqcup_{n\in \mathbb{N}} T_n$ which is directed by inclusion, i.e. $\lambda\leq \lambda'$ if $\lambda\subset \lambda'$. If $\lambda\leq \lambda'$ then we have a map $f_{\lambda,\lambda'}:\text{Hol}_{\gamma}(\lambda)\to \text{Hol}_{\gamma}(\lambda')$ defined by adding identities $\rho_i=\text{id}$ and splitting $\gamma_i$ in two parts, at all points of $\lambda'$ that are not in $\lambda$. The anafunctor $\text{Hol}_{\gamma}$ is the direct limit of the direct system of sets $\{\text{Hol}_{\gamma}(\lambda)\}_{\lambda\in T}$. The anchor maps and the actions defined on each $\text{Hol}_{\gamma}(\lambda)$ descend to $\text{Hol}_{\gamma}$.

\begin{proposition}
The smooth manifold $\text{Hol}_{\gamma}$ together with the anchor maps $\alpha_l$ and $\alpha_r$, the actions $\text{Mor}(\mathcal{P}_{\gamma(0)})_{ \ s}\times_{\alpha_{l}}\text{Hol}_{\gamma}(\lambda)\to \text{Hol}_{\gamma}(\lambda)$, $\text{Hol}_{\gamma}(\lambda)_{ \ \alpha_{r}}\times_{t} \ \text{Mor}(\mathcal{P}_{\gamma(1)})\to \text{Hol}_{\gamma}(\lambda)$ and $ \text{Hol}_{\gamma}(\lambda)\times \text{Mor}(\Gamma)\to  \text{Hol}_{\gamma}(\lambda)$, define a $\Gamma$-equivariant anafuntor $\text{Hol}_{\gamma}:\mathcal{P}_{\gamma(0)} \to \mathcal{P}_{\gamma(1)}$.
\end{proposition}

Before proceeding to the construction of the 2-dimensional holonomies we need the following definition:
\begin{definition}
  A smooth bigon in $M$ is a smooth map $\Sigma:I^2\to M$ such that $\Sigma(s,0)=x$ and $\Sigma(s,1)=y$ for all $s\in I$. In other words, a bigon is a fixed-ends homotopy between the paths $\gamma(t)=\Sigma(0,t)$ and $\gamma'(t)=\Sigma(1,t)$. In this case we write $\Sigma:\gamma\Rightarrow\gamma'$.\\

  A bigon $\Sigma:\gamma\Rightarrow \gamma'$ is called small, if there exist $n\in \mathbb{N}$, $\lambda\in T_n$ and sections $\sigma_i:U_i\to \text{Obj}(\mathcal P)$ defined on open sets $U_i$ such that
  \[
  \Sigma(\{(s,t) \ | \ t_{i-1}\leq t \leq t_i, 0\leq s\leq 1\})\subset U_i.
  \]
\end{definition}
To define the holonomy along a 2-simplex $\sigma:\Delta_2\to M$ we first consider the bigon $\Sigma=\sigma\circ\Theta_2:\gamma\to\gamma'$. The idea is to subdivide $\Sigma$ in small bigons $\Sigma_i$ and then define a $\Gamma$-equivariant transformation $\varphi_{\Sigma_i}^{\text{small}}:\text{Hol}_{\gamma_{i-1}}\to \text{Hol}_{\gamma_i}$ between the parallel transports along $\gamma_{i-1}$ and $\gamma_i$, for each small bigon $\Sigma_i$ in the subdivision of $\Sigma$. Finally the $\Gamma$-equivariant transformation for $\Sigma$ is defined as the composition $\varphi_{\Sigma}(s):=\varphi_{\Sigma_n}^{\text{small}}\circ \cdots \circ\varphi_{\Sigma_1}^{\text{small}}$.

\begin{definition}
Let $\Sigma:\gamma \Rightarrow \gamma'$ be a bigon and $\xi\in \text{Hol}_{\gamma}$. A \emph{horizontal lift} of $\Sigma$ with source $\xi$ is a tuple $(n,\lambda,\{\Phi_i\}_{i=1}^n,\{\rho_i\}_{i=1}^n,\{g_i\}_{i=1}^n)$ consisting of $n\in \mathbb{N}$, a subdivision $\lambda\in T_n$ and smooth maps
\begin{itemize}
	\item $\Phi_i:[0,1]\times [t_{i-1},t_i]\to \text{Obj}(\mathcal P)$,
	\item $\rho_i:[0,1]\to \text{Mor}(\mathcal P)$ with $\rho_0$ and $\rho_n$ constant,
	\item $g_i:[0,1]\to G$ with $g_i(0)=1$,
\end{itemize}

such that the following conditions are satisfied:
\begin{enumerate}
	\item $\Phi_i$ is a lift of $\Sigma$, i.e., $\pi\circ \Phi_i=\Sigma|_{[0,1]\times [t_{i-1},t_i]}$ for all $1\leq i\leq n$.
	\item $t(\rho_i(s))=\Phi_{i+1}(s,t_i)$ for all $0\leq i < n$ and $s(\rho_i(s))=R(\Phi_i(s,t_i),g_i(s))$ for all $1\leq i \leq n$.
	\item The paths $\gamma'_i(t):=\Phi_i(1,t)$, $\nu_i(s):=\Phi_i(s,t_{i-1})$ and $\rho_i$ are horizontal for all $1\leq i \leq n$.
	\item $\xi=\rho_n\ast\gamma_n\ast\cdots\ast\gamma_1\ast \rho_0$ with $\gamma_i(t):=\Phi_i(0,t)$ and $\rho_i:=\rho_i(0)$.
\end{enumerate}
\end{definition}

\begin{lemma}
For every small bigon $\Sigma:\gamma\Rightarrow \gamma'$ and every $\xi\in \text{Hol}_{\gamma}$ there exists a horizontal lift with source $\xi$.
\end{lemma}

Finally, for an arbitrary bigon $\Sigma:\gamma\Rightarrow \gamma'$ there exists a subdivision $s\in T_n$ such that the pieces $\Sigma_i(s,t):=\Sigma((s_i-s_{i-1})s+s_{i-1},t)$ are small. Then we define
\[\varphi_{\Sigma}(s):=\varphi_{\Sigma_n}^{\text{small}}\circ\cdots\circ \varphi_{\Sigma_1}^{\text{small}}.\]

The map $\varphi_{\Sigma}(s)$ is independent of the choice of $s$.

\begin{proposition}
The map $\varphi_{\Sigma}:\text{Hol}_{\gamma}\to \text{Hol}_{\gamma'}$ is a $\Gamma$-equivariant transformation.
\end{proposition}

\begin{proposition}
Let $\mathcal{P}$ be a principal $\Gamma$-$2$-bundle with flat connection $\Omega$. Then the assignments $x\mapsto \mathcal{P}_x, [\gamma]\mapsto \text{Hol}_{[\gamma]}$, and $[\Sigma]\mapsto \varphi_{[\Sigma]}$ form a $2$-functor
\[
\text{tra}_{\mathcal{P}}:\pi_{\leq 2}(M)\to \Gamma\text{-Tor}(\mathcal P).
\]
\end{proposition}
The equivalence between the local and global definitions of holonomies may be found in section 5 of \cite{waldorf2}. The main observation is that the global construction of holonomies relies on lifting paths and surfaces horizontally, which is done locally by solving the differential equations (\ref{poe}) and (\ref{soe}).\\

   The construction above gives a map $T:2\text{-}\mathcal Bun^f{_\mathcal{P}}(M)\to \text{Rep}_{\leq2}(M,\mathcal{P})$ the map from flat principal 2-bundles to representations of the fundamental 2-groupoid of $M$.

\section{Higher local systems}

In this section we will review the higher Riemann Hilbert correspondence of \cite{Block}. We will follow the exposition in \cite{abadrham}.
Let  $E=\bigoplus_k E^k$ be a $\mathbb Z$-graded vector bundle over $M$. A form $\omega\in\Omega^n(M,E^k)$ has form-degree $n$ and inner-degree $k$. The total degree of $\omega$ is the sum of its form and inner degrees and is denoted $|\omega|=n+k$.
\begin{definition}
  A $\mathbb Z$-graded connection on a graded vector bundle $E\to M$ is a linear operator $D:\Omega(M,E)\to \Omega(M,E)$ that increases the total degree by one and satisfies the Leibniz rule
  \[
  D(\alpha\wedge\omega)=d\alpha\wedge\omega+(-1)^{|\alpha|}\alpha\wedge D(\omega),
  \]
  for homogeneous $\alpha\in\Omega(M)$ and $\omega\in\Omega(M,E)$. The connection is flat if $D^2=0$. \end{definition}
An $\infty$-local system over $M$ is a graded vector bundle together with a flat $\mathbb{Z}$ graded connection.
There is a differential graded category $\Loc_\infty(M)$  whose objects are $\infty$-local systems and whose morphisms are the graded vector spaces:
  \[
  \text{Hom}_{\Loc_\infty(M)}(E,E')=\Omega(M,\text{Hom}(E,E'))
  \]
  with differential:
  \[ \overline{D}(\eta)=D' \circ \eta -(-1)^{|\eta|} \eta\circ  D.\]
On an open subset  $U\subseteq M$ where the bundle is trivial, a flat $\ZZ$-graded connection can be written as $d-\omega$ where $\omega\in\Omega(U,\End(E))$ is a Maurer-Cartan form of total degree 1. This means that $\omega=\sum_i\omega_i$ with $\omega_i\in\Omega^i(U,\End^{1-i}(E))$ and the Maurer-Cartan equation is satisfied $d\omega=\omega\wedge\omega$. Notice that in degree zero, the Maurer-Cartan equation reads $\omega_0\wedge\omega_0=0$. Since $\omega_0\in\Omega^0(U,\End^1(E))$, this means that $\omega_0$ endows the fibers of $E$ with a cochain complex structure. Therefore we will write $\partial=\omega_0$.

\begin{definition}
A cochain complex bundle is a graded vector bundle $p:E\to M$ with a morphism of bundles $\partial:E^\bullet\to E^{\bullet+1}$ such that $\partial^2=0$. The local trivializations are required to be isomorphisms of complexes of vector bundles on each fibre, i.e., there is a complex $(V,\partial)$ such that if $E$ is trivial over a certain open $U\subset M$, then the trivialization map $\varphi_U:p^{-1}(U)\to U\times V$ restricts to an isomorphism of complexes $\varphi_x:E_x\to (V,\partial)$.
\end{definition}
\begin{lemma}\label{loccomplex}
Let $(E,D)$ be an $\infty$-local system over a connected manifold $M$ of dimension $n$ with projection $\pi:E\to M$. Then $E$ is a cochain complex bundle.
\end{lemma}
\begin{proof}
For an arbitrary $x_0\in M$, choose a neighbourhood $U\subset M$ over which $E$ is trivial and a coordinate chart $\psi:U\to\mathbb{R}^n$ that maps $x_0$ to the origin. Over $U$ the connection takes the form $D=d-\partial-\cdots$, where $\partial \in \Omega^0(U, \text{End}^1(E|_{U}))$. Furthermore, since the connection is flat we have that $\partial^2=0$, which means that the fibers have a chain complex structure with differential $\partial$. Now for any other $x\in U$ let $\gamma_x$ be a path in $U$ connecting $x$ to $x_0$ such that $\psi\gamma_x$ is a radial path connecting $\psi(x)$ to the origin. The parallel transport along these paths provides isomorphisms of cochain complexes $T_{\gamma_x}:E_x\to E_{x_0}$. A local trivialisation as a cochain complex bundle over $U$ is defined by $\varphi:\pi^{-1}(U)\to U\times E_{x_0}$, $\varphi(p)=(\pi(p),T_{\gamma_{\pi(p)}}(p))$.
The fact that $M$ is connected guarantees that all the fibers over $M$ are isomorphic as cochain complexes.
\end{proof}
For a fixed cochain complex bundle $(E,\partial)$ we denote by $\Loc_\infty^{(E,\partial)}(M)$ the category of $\infty$-local systems over $M$ with underlying cochain complex bundle $(E,\partial)$.\\
\begin{definition}
  Let $(E,\partial)$ be a cochain complex bundle over $M$. The linear 2-groupoid of $(E,\partial)$ is the 2-category $GL(E,\partial)$ such that
  \begin{itemize}
    \item The objects are the fibers $(E_x,\partial_x)$ for $x\in M$.
    \item The 1-morphisms are cochain complex morphisms $(E_x,\partial_x)\to(E_y,\partial_y)$ with the usual composition.
    \item The 2-morphisms are (algebraic) homotopies between morphisms of complexes modulo exact maps. The horizontal and vertical compositions
    are the usual ones for homotopies.
  \end{itemize}
  A linear representation of $\pi_{\leq2}(M)$ on $(E,\partial)$ is a 2-functor $\mathcal R:\pi_{\leq2}(M)\to GL(E,\partial)$. Linear representations of $\pi_{\leq2}(M)$ on $(E,\partial)$ will be denoted $\text{Rep}_{\leq2}(M,(E,\partial))$.
\end{definition}

Recall that the infinity groupoid of a manifold $M$ is the simplicial set
\[ \pi_\infty(M):= \mathsf{Sing}(M),\]
of smooth singular chains on $M$.
So far we have defined infinitesimal higher local systems. The Riemann-Hilbert correspondence is about comparing this notion with a global one, represesentations of the infinity groupoid. \\

Given a simplicial set $X_{\bullet}$ with face and degeneracy maps denoted by
\[d_i: X_k\rightarrow X_{k-1} \quad \text{and} \quad s_i: X_k \rightarrow X_{k+1},\]
respectively.  We will use the notation
\begin{eqnarray*}
\front_i&:=&(d_0)^{k-i}:X_k \rightarrow X_i,\\
\back_i&:=&d_{i+1} \circ \dots \circ d_{k} :X_k \rightarrow X_i,
\end{eqnarray*}
for the maps that send a simplex to its $i$-th back and front face.
The $i$-th vertex of a simplex $\sigma \in X_k$
will be denoted $v_i(\sigma)$, or simply $v_i$, when no confusion can arise.
In terms of the above operations, one can write
\[v_i=(\front_0\circ \back_{i})(\sigma). \]

A cochain $F$ of degree $k$ on $X_{\bullet}$ with values in an algebra $\A$ is a map
\begin{align*}
F: X_k \to \A.
\end{align*}
As usual, the cup product of two cochains  $F, F'$ of degree $i$ and $j$, respectively, is the cochain of degree $i+j$ defined by the formula:
\[(F \cup F') (\sigma):= F(\back_i(\sigma)) F'(\front_j(\sigma)).\]

A representation of $X_{\bullet}$ consists of the following data:
\begin{enumerate}
\item A graded vector space $E_x=\bigoplus_{k\in \mathbb{Z}}E_x^k$,
for each zero simplex $x \in X_0$.
\item A sequence of operators $\{F_k\}_{k\geq 0}$, where $F_k$ is a $k$-cochain that assigns to $\sigma \in X_k$ a linear map
\[F_k(\sigma)\in \mathrm{Hom}^{1-k}(E_{v_k(\sigma)},E_{v_0(\sigma)}).\]
\end{enumerate}
These operators are required to satisfy the following  equations:
\begin{equation}\label{structure equations}
\sum_{j=1}^{k-1} (-1)^{j} F_{k-1}(d_j (\sigma))  +
\sum_{j=0}^{k} (-1)^{j+1} (F_{j}\cup F_{k-j})(\sigma)=0.
\end{equation}

Given a manifold $M$, there is a differential graded category $\Rep(\pi_\infty(M))$ whose objects are representations of the infinity groupoid of $M$.
The higher Riemann-Hilbert correspondence proved in \cite{Block}, is the following result.

\begin{theorem}\label{RH}
There is an $\mathsf{A}_\infty$ functor \[ {\cal{I}}:\Loc_\infty(M) \rightarrow \Rep(\pi_\infty(M))\]
which is a quasi-equivalence of differential graded categories.
\end{theorem}

The theorem above is proved by providing an explicit formula for the generalized holonomies associated to simplices of all dimensions.
In the local case, where the superconnection takes the form $D=d-\omega$, the holonomy $\mathsf{hol}(\sigma)$ associated to a simplex
$\sigma: \Delta_k \rightarrow M$ is given by:

\[
\mathsf{hol}(\sigma)=\sum_{n\geq0}(-1)^{n}\int_{\Delta_{n}\times I^{k-1}}(\Theta_k)_{(n)}^*\left(p_1^*(\sigma^*(\omega))\wedge \cdots \wedge p_{n}^*(\sigma^*(\omega))\right),
\]

where $p_i: (\Delta_k)^n \rightarrow \Delta_k$ is the $i$-th projection, $\Theta_k : I^{ k} \rightarrow \Delta_k$ is certain piecewise linear map folding the cube onto the simplex  and  $(\Theta_k)_{(n)}: \Delta_n \times I^{k-1} \rightarrow (\Delta_k)^n$ is given by:
\[ (\Theta_k)_{(n)}(t_1, \dots, t_n, x_2, \dots, x_k)=(\Theta_k(t_1, x_2,\dots, x_k),\dots, \Theta_k(t_n, x_2, \dots, x_k)).\]

We observe that, for $k=2$ the map $\Theta_2$ coincides with the one defined in Remark \ref{Theta2}.

\begin{lemma}\label{map2}
  There is a map ${\cal{T}}:\Loc_\infty^{(E,\partial)}(M)\to \text{Rep}_{\leq2}(M,(E,\partial))$ for any fixed cochain complex bundle $(E,\partial)$.
\end{lemma}
\begin{proof}
Geometrically, this corresponds to the fact that the product of a 2-simplex with an interval is the union of three 3-simplices:
\begin{figure}[h!]\label{simplices}
  \begin{center}
  \includegraphics[scale= 0.2]{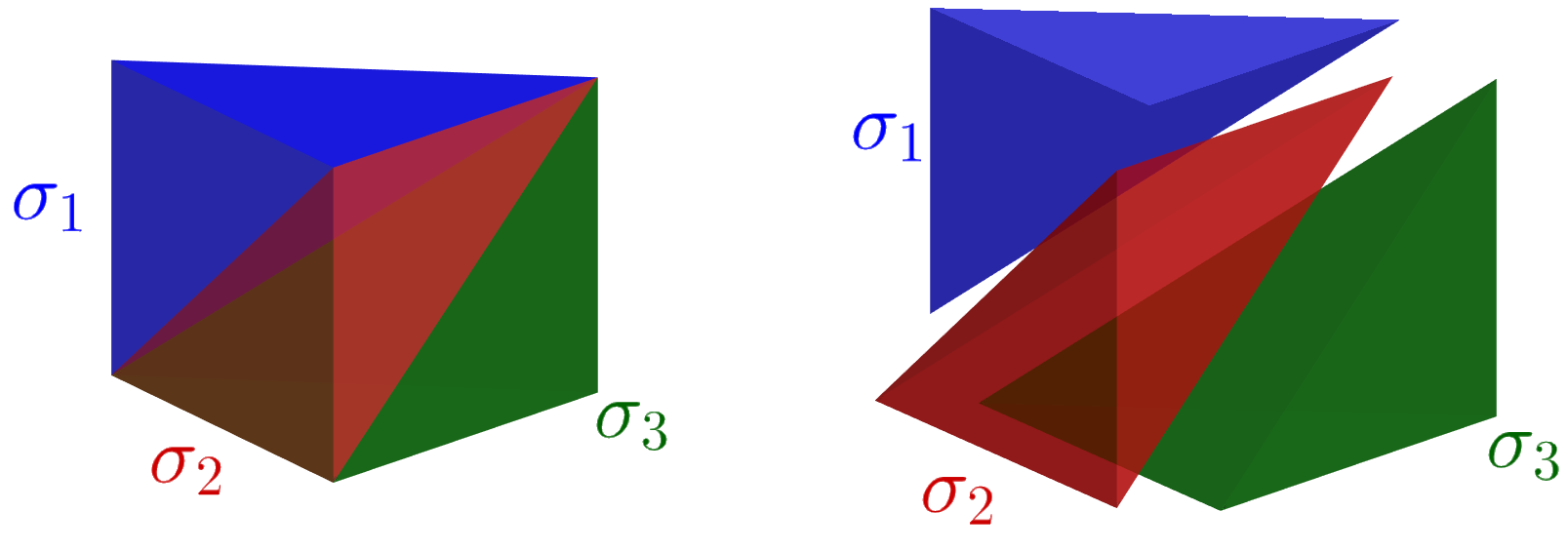}
  \end{center}
\end{figure}

Given a local system $ E \in \Loc_\infty(M)$, Theorem \ref{RH} provides a representation ${\mathcal{I}}(E)$ of $\pi_\infty(M)$. This representation assigns holonomies
 $\mathsf{hol}(\sigma)$ to simplices $\sigma: \Delta_k \rightarrow M$ of arbitrary dimensions. We define the functor $\mathcal{R}: \pi_{\leq2}(M) \rightarrow GL(E,\partial)$ as follows. To a point $x \in M$ we associate the cochain complex $(E_x, \partial)$. To a path, the holonomy with respect to the connection. To the class of a simplex $\sigma: \Delta_2 \rightarrow M$ we associate the holonomy $\hol(\sigma)$ given by ${\mathcal{I}}(E)$. We need to prove that up to exact terms, this does not depend on the representative of the class. Suppose $H: \Delta_2 \times I \rightarrow M$ is a homotopy with fixed ends between simplices $\sigma$ and $\sigma'$. For $i=1,2,3$ we define the simplices $\sigma_i: \Delta_3 \rightarrow M$ by the formulas:
 \[ \sigma_1(t_1,t_2,t_3)=H(t_2,t_3,t_1); \,\, \sigma_2(t_1,t_2,t_3)=H(t_1,t_3,t_2); \,\, \sigma_3(t_1,t_2,t_3)=H(t_1,t_2,t_3).\]
 Using relations (\ref{structure equations}) for the simplices $\sigma_1, \sigma_2, \sigma_3$, together with the fact that holonomies vanish on degenerate simplices (proved in Proposition 3.26 of \cite{abadrham})
 one obtains that $\mathsf{hol}(\sigma)$ coincides with $\mathsf{hol}(\sigma')$ modulo exact terms. This proves that the map $\mathcal{R}$ is well defined. The fact that it defines a functor follows  from (\ref{structure equations}).
\end{proof}

\section{A comparison result}
The first step towards a comparison is a procedure to construct a principal 2-bundle with flat connection out of an $\infty$-local system. At the level of objects the construction is merely the frame bundle construction, which requires only the vector bundle structure of the $\infty$-local system. To perform the construction at the level of morphisms we will rely on the cochain complex structure of the fibers provided by the flat $\mathbb{Z}$-graded connection. It is worth noting that this construction yields a simple class of principal 2-bundles.
By specializing the construction of the general linear 2-groupoid of a cochain complex bundle to the case where $M$ is a point one obtains
a Lie crossed module  $\Gamma(V,\partial)=(H,G,\tau,\alpha)$ associated to a cochain complex $(V,\partial)$. The corresponding differential crossed module
has been computed in \cite{faria22}.

\begin{lemma}\label{Faria2}
  Let $(V,\partial)$ be a cochain complex. The differential Lie crossed module associated to $\Gamma(V,\partial)$ is $\gamma(V,\partial)=(\mathfrak h,\mathfrak g,\tau_*,\alpha_*)$ where:
\begin{enumerate}
    \item $\mathfrak g=\text{gl}^0(V)$ is the vector space of degree zero cochain maps $V\to V$, endowed with the commutator bracket.
    \item $\mathfrak h=\text{gl}^{-1}(V)$, with
    \[
    \text{gl}^{-1}(V)=\frac{\text{End}^{-1}(V)}{[\partial,\text{End}^{-2}(V)]}.
    \]
    This quotient is endowed with the Lie bracket
    \[
    [T,S]=S\partial T-T\partial S+ST\partial-TS\partial;
    \]
    and $[\partial,\text{End}^{-2}(V)]$ is the ideal generated by elements of the form $\partial h-h\partial$.
    \item $\tau_*:\mathfrak h\to \mathfrak g$ is given by $\tau_*(S):=\tau S+S\tau$.
    \item The action $\alpha_*$ is such that $\alpha_*(R)(S)=RS-SR$.
\end{enumerate}
\end{lemma}

  When we interpret the crossed module $\Gamma(V,\partial)$ as a Lie 2-group, we have that a pair $(g,h):g\to g'$ is a morphism if $g'=\tau(h)g$. Unraveling this equation leads to $g'-g=[\partial,hg]$, which means that $hg$ is a homotopy from $g'$ to $g$.

\begin{lemma}\label{2frame}
  Given a cochain complex bundle $(E,\partial)$, there is an associated principal 2-bundle $\fr$ which we call the frame 2-bundle.
\end{lemma}
\begin{proof}
  Consider a bundle of complexes $(E,\partial)\to M$ and a point $x_0\in M$. We let $(V,\partial)=(E_{x_0},\partial_{x_0})$ and $\Gamma=\Gamma(V,\partial)=(H,G,\tau,\alpha)$. We define a groupoid $\fr$ and an action $\fr\times\Gamma\to\fr$ as follows:
\begin{enumerate}
    \item For every $x\in M$, let
\[\text{Obj}(\fr)_x:=\{\rho:V\to E_x\;\big|\; \rho \text{ is an isomorphism of complexes of degree }0\}.\]
We set $\text{Obj}(\fr)=\bigsqcup_{x\in M}\text{Obj}(\fr)_x$. The projection $\pi:\text{Obj}(\fr)\to M$ is such that $\pi^{-1}(x)=\text{Obj}(\fr)_x$.

    \item Let $\rho,\rho'\in \text{Obj}(\fr)_x$. Define the set
\[\text{Hom}_{\fr}(\rho,\rho'):=\{\xi:V\to E_x\;\big| \; \partial_x\xi+\xi\partial=\rho-\rho'\}/\sim\]
where $\xi\sim\xi'$ if they are homotopic. So morphisms between $\rho$ and $\rho'$ are homotopy classes of homotopies between them. The fact that being homotopic is an equivalence relation provides the identity morphisms (reflexiveness), every homotopy is invertible (symmetry) and a way to compose homotopies (transitiveness). The source and target maps are denoted $s$ and $t$ respectively. Given a homotopy $\xi\in\text{Hom}_{\fr}(\rho,\rho')$ for  $\rho,\rho'\in \text{Obj}(\fr)_x$, we define $\pi(\xi)=\text{id}_x$.
\end{enumerate}
Next we define an action $R:\fr\times\Gamma\to\fr$. On objects it is defined by composition: $R:\text{Obj}(\fr)\times G\to\text{Obj}(\fr)$ is such that $(\rho,g)\mapsto R(\rho,g)=\rho g$. At this level the action is clearly free and transitive along the fibres. Suppose $\xi:\rho\to\rho'$, then we have that $(\xi,(h,g)):(\rho,g)\to(\rho',\tau(h)g)$ is a morphism in $\fr\times\Gamma$. We define $R:\text{Mor}(\fr)\times(H\ltimes G)\to\text{Mor}(\fr)$ as
\[(\xi,(h,g))\mapsto R(\xi,(h,g))=\xi\tau(h)g-s(\xi)hg.\]
The following equation shows that $ R(\xi,(h,g))$ is indeed a morphism $\rho g\to\rho'\tau(h)g$:
\begin{align*}
[\partial,  R(\xi,(h,g))]=&[\partial,\xi\tau(h)g-\rho hg] \\
=&[\partial,\xi]\tau(h)g-\rho[\partial,h]g\\
=& \rho g-\rho'\tau(h)g.
\end{align*}
Let us check that the action just defined makes $\fr$ a principal $\Gamma$-2-bundle, i.e. the functor $(\text{pr},R):\fr\times\Gamma\to\fr\times_M \fr$ is a weak equivalence. Since the action is transitive along the fibres on objects, the functor $(\text{pr},R)$ is surjective. Next we check that the functor is fully faithful: consider objects $(\rho,g)$ and $(\rho',g')$ in $\fr\times\Gamma$. In order to have morphisms $(\rho,g)\to(\rho',g')$ there must be at least one $h\in H$ such that $\tau(h)g=g'$. If this is the case we have a function
\[
(\text{pr},R):\text{Hom}_{\fr\times\Gamma}((\rho,g),(\rho',g'))\to \text{Hom}_{\fr\times_M\fr}((\rho,\rho g),(\rho',\rho'g')).
\]
The function is injective: take morphisms $(\xi,(h,g))$ and $(\xi',(h',g))$, and suppose they map to the same morphism, this is $(\xi,\xi\tau(h)g-s(\xi) hg)=(\xi',\xi'\tau(h')g-s(\xi') h'g)$. Clearly we must have $\xi=\xi'$. The equation from the second component may be rewritten as
\[\xi(\tau(h)-\tau(h'))=s(\xi)(h-h')\]
Next we notice that, since $\tau(h)g=\tau(h')g$ we have $\tau(h)=\tau(h')$, therefore we conclude that $h=h'$ and the function is injective.\\
The function is surjective: take a pair of homotopies $\xi:\rho\to\rho'$ and $\sigma:\rho g\to\rho'g'$. Define $h=\rho^{-1}\xi g'g^{-1}-\rho^{-1}\sigma g^{-1}$. We have
\begin{align*}
\tau(h)g=&[\partial,\rho^{-1}\xi g'-\rho^{-1}\sigma]+g\\
=&\rho^{-1}[\partial,\xi]g'-\rho^{-1}[\partial,\sigma]+g\\
=&\rho^{-1}(\rho-\rho')g'-\rho^{-1}(\rho g-\rho' g')+g=g'.
\end{align*}
Hence $(\xi,(h,g))\in\text{Hom}_{\fr\times\Gamma}((\rho,g),(\rho',g')) $. Furthermore, we have
\[
R(\xi,(h,g))=\xi\tau(h)g-\rho h g=\xi g'-\xi g'+\sigma=\sigma,
\]
proving surjectivity.
\end{proof}
\begin{remark}\label{trivialcocycle}
  As $E$ is a cochain complex bundle, the usual cocycles $g_{ij}$ defined from local trivialisations of the bundle take values in the group $G=GL^0(V)$. Setting $a_{ijk}=0\in H=GL^{-1}(V)$ is easy to check that we have a $\Gamma(V,\partial)$-cocycle $(g,a)$. The principal 2-bundle built from this cocycle using Lemma \ref{cocycle1} is precisely the frame 2-bundle $\fr$, making clear the statement that the frame 2-bundle construction yields only a simple class of principal 2-bundles.
\end{remark}
From Proposition \ref{loccomplex} and Lemma \ref{2frame} we get that every $\infty$-local system has an associated principal 2-bundle. Now we aim to build a flat connection on the frame 2-bundle from an $\infty$-local system.
\begin{theorem}\label{connection}
  Let $(E,D)$ be an $\infty$-local system over a connected smooth manifold $M$. Then there is a flat connection $\Omega_D$ defined on the frame 2-bundle $\fr$.
\end{theorem}
\begin{proof}
  First let us fix a point $x_0\in M$ and call $(V,\partial)=(E_{x_0},\partial_{x_0})$, the fiber over $x_0$ with its cochain complex structure. We will define a differential $\Gamma(V,\partial)$-cocycle from the $\ZZ$-graded connection $D$. Let $\{U_i\}_{i\in I}$ be an open covering of $M$ such that the vector bundle $E$ trivialises as a cochain complex bundle over each $U_i$ with fiber $(V,\partial)$. Let $(g,a)$ be $\Gamma(V,\partial)$-cocycle defined in remark \ref{trivialcocycle}. For differential components of the cocycle consider the local form of $D$ over $U_i$ which is $D=d+\sum_{k\geq0}\omega^k_i$ where $\beta^0_i=\partial\in\Omega^0(U_i,\End^1(V))$ and $\beta_i^k\in\Omega^k(U_i,\End^{1-k}(V))$. Let $\text{pr}:\End^{-1}(V)\to\text{End}^{-1}(V)/[\partial,\text{End}^{-2}(V)]$ be the natural projection to the quotient, then we define
  \[
  A_i=\omega^1_i\in\Omega^1(U_i,\End^0(V)),\quad B_i=-\text{pr}\circ\omega^2_i\in\Omega^2(U_i,\mathfrak h).
  \]
  The Maurer-Cartan equation for $\omega_i=\sum_{k\geq0}\omega^k_i$ has the following implications:
  \begin{itemize}
    \item The term of degree one of the equation is $\partial \omega^1_i=\omega^1_i\partial$, so $\omega^1_i$ actually takes values in the algebra of cochain maps and $A_i=\omega^1_i\in\Omega^1(U_i,\mathfrak g)$.
    \item The term of degree two is $\partial \omega^2_i+\omega^2_i\partial=d\omega^1_i+\omega^1_i\wedge \omega^1_i$. Projecting to the quotient and rewriting appropriately yields $\tau_*(B_i)=dA_i+[A_i,A_i]/2$, which means that the fake curvature of the $\Gamma(V,\partial)$-cocycle vanishes.
    \item The term of degree three is $d\omega^2_i+\omega^1_i\omega^2_i-\omega^2_i\omega^1_i=\partial\omega^3_i-\omega^3_i\partial$. Projecting to the quotient and rewriting we get $dB_i+\alpha_*(A_i)(B_i)=0$, thus the 3-curvature also vanishes.
  \end{itemize}
  The final component of the cocycle is the form $\varphi_{ij}\in\Omega^1(U_i\cap U_j,\mathfrak h)$, this one we take to be zero $\varphi_{ij}=0$. Now suppose that $U_i\cap U_j$ is non-empty, then the forms $\omega_i$ and $\omega_j$ are gauge equivalent via the cocycle $g_{ij}$, this is $\omega_i=g_{ij}\omega_j g_{ji}+g_{ij}dg_{ji}$. After projecting onto the quotient, the terms of degree one and two of the gauge equivalence equation may be written respectively as
  \[
  A_i=g_{ij}A_jg_{ji}+g_{ij}dg_{ji},\quad\text{and}\quad B_i=g_{ij}B_jg_{ji}.
  \]
  The previous equations are the conditions (\ref{diffcocycle1}) and (\ref{diffcocycle2}) required of a differential cocycle. Condition (\ref{diffcocycle3}) is satisfied trivially. The proof is completed by applying Lemma \ref{cocycle2}.
\end{proof}

 Fix a cochain complex bundle $(E,\partial)$. Let $\Loc_\infty^{(E,\partial)}(M)$ denote the category of $\infty$-local systems over $M$ such that their cochain complex bundle is isomorphic to $(E,\partial)$. We have defined a map
 \begin{align*}
 K:\Loc_\infty^{(E,\partial)}(M)&\to 2\text{-}\mathcal Bun_{\fr}^f(M).\\
 (E,D)&\mapsto (\fr,\Omega_D)
 \end{align*}

\begin{lemma}\label{linear2group}
  Let $(E,\partial)$ be a cochain complex bundle over $M$ and fix a basepoint $\ast \in M$. Then there is a map $\mathcal{K}:\text{Rep}_{\leq2}(M,(E,\partial))\to\text{Rep}_{\leq2}(M, \fr)$.
\end{lemma}
\begin{proof}
  Suppose $\Rcal:\pi_{\leq2}(M)\to\gl$ is a linear representation, we will define the functor
\[
\mathcal{K}\Rcal:\pi_{\leq2}(M)\to\Gamma\text{-Tor}(\fr).
\]
For every point $x\in M$, the fiber $\fr_x$ is a $\Gamma$-torsor, so we set $\mathcal{K}\Rcal(E_x)=\fr_x$.\\

If $\rho:x\to y$ is a path connecting $x$ to $y$, then $\Rcal(\rho):E_x\to E_y$ is a morphism of complexes. The push forward $\Rcal(\rho)_*:\fr_x\to\fr_y$ given by composition to the left with $\Rcal(\rho)$ is a $\Gamma$-equivariant functor, hence it induces a $\Gamma$-equivariant anafunctor and we define $\mathcal{K}\Rcal(\rho)=\Rcal(\rho)_*$. It can be checked easily that the functor preserves composition of paths.\\

Finally take paths $\rho,\rho':x\to y$ and let $\xi:\rho\to\rho'$ be a homotopy between them. Then $\Rcal(\xi):\Rcal(\rho)\to\Rcal(\rho')$ is an (algebraic) homotopy. The homotopy $\Rcal(\xi)$ induces a natural transformation between the functors $\Rcal(\rho)_*,\Rcal(\rho')_*:\fr_x\to\fr_y$ as follows: let $\alpha:V\to E_x$ be an object in $\fr_x$, then $\Rcal(\xi)_*(\alpha)=\Rcal(\xi)\circ\alpha:\Rcal(\rho)_*(\alpha)\to\Rcal(\rho')_*(\alpha)$ is a homotopy. Let us check that the push forward $\Rcal(\xi)_*$ is indeed a natural transformation: consider a morphism $\mathcal A:\alpha\to\beta$ in $\fr_x$, then we have the following diagram
\[
\begin{CD}
\Rcal(\rho)_*(\alpha)@>\Rcal(\rho)_*(\mathcal A)>>\Rcal(\rho)_*(\beta)\\
@V\Rcal(\xi)_*(\alpha)VV@VV\Rcal(\xi)_*(\beta)V\\
\Rcal(\rho')_*(\alpha)@>>\Rcal(\rho')_*(\mathcal A)> \Rcal(\rho')_*(\beta)
\end{CD}
\]
A simple computation shows that
\[
[\Rcal(\xi)_*(\mathcal A), \partial]=(\Rcal(\xi)\circ\alpha+\Rcal(\rho')\circ\mathcal A)-(\Rcal(\rho)\circ\mathcal A+\Rcal(\xi)\circ\beta),
\]
which shows that the diagram is commutative up to homotopy. Once again, since the natural transformation is defined by composition to the left, we can see that it is $\Gamma$-equivariant. We define $\mathcal{K}\Rcal(\xi)=\Rcal(\xi)_*$. The fact that $\mathcal{K}\Rcal$ preserves the horizontal and vertical composition of 2-morphisms derives from the same fact for $\Rcal$ and the properties of the push forward.
\end{proof}
 \begin{lemma}
    If $h:[0,1]\to H$ is a solution to equation (\ref{soe}), then
    \[
    h(1)=\left(\int_0^1\int_0^1\alpha(g_{\Sigma_s}(t)^{-1})_*\left(B\left(\frac{\partial}{\partial t}\Sigma_s(t),\frac{\partial}{\partial s}\Sigma_s(t)\right)\right)\text{Hol}_{\Sigma_s}^{-1}dtds\right)\text{Hol}_{\Sigma_1}.
    \]
  \end{lemma}
  \begin{proof}
    Equation (\ref{soe}) may be rewritten as
    \begin{align*}
      \dfrac{dh(s)}{ds}=&\left(L_{h(s)}\right)_{*}\left(\int_0 ^1 \alpha\left(g_{\Sigma_s(t)^{-1}}\right)_{*}\left(B\left(\dfrac{\partial}{\partial t} \Sigma_s(t), \dfrac{\partial}{\partial s} \Sigma_s(t) \right)\right)dt\right)\\
      =&\int_0 ^1 \alpha\left(g_{\Sigma_s(t)^{-1}}\right)_{*}\left(B\left(\dfrac{\partial}{\partial t} \Sigma_s(t), \dfrac{\partial}{\partial s} \Sigma_s(t) \right)\right)dt\\
      &+h(s)\left(\int_0^1\text{Ad}_{g_{\Sigma_s}(t)^{-1}}\tau_*(B)\left(\dfrac{\partial}{\partial t} \Sigma_s(t), \dfrac{\partial}{\partial s} \Sigma_s(t) \right)dt\right)\\
      =&\int_0 ^1 \alpha\left(g_{\Sigma_s(t)^{-1}}\right)_{*}\left(B\left(\dfrac{\partial}{\partial t} \Sigma_s(t), \dfrac{\partial}{\partial s} \Sigma_s(t) \right)\right)dt+h(s)\text{Hol}_{\Sigma_s}^{-1}\frac{d\text{Hol}_{\Sigma_s}}{ds}.
    \end{align*}
    Therefore we get the following equation
    \[
    \frac{d\left(h(s)\text{Hol}_{\Sigma_s}^{-1}\right)}{ds}=\left(\int_0 ^1 \alpha\left(g_{\Sigma_s(t)^{-1}}\right)_{*}\left(B\left(\dfrac{\partial}{\partial t} \Sigma_s(t), \dfrac{\partial}{\partial s} \Sigma_s(t) \right)\right)dt\right)\text{Hol}_{\Sigma_s}^{-1}.
    \]
    Integrating the previous equation leads to the desired result.
  \end{proof}
  \begin{lemma}
    Define
    \[
    a_s(t):=A_{\Sigma(s,t)}\left(\frac{\partial\Sigma}{\partial t}\right)\quad \text{and}\quad b_s(t):=B_{\Sigma(s,t)}\left(\dfrac{\partial}{\partial t} \Sigma_s(t), \dfrac{\partial}{\partial s} \Sigma_s(t)\right).
    \]
    Then the integral
    \begin{equation}\label{last}
      \int_0^1\int_0^1\alpha(g_{\Sigma_s}(t)^{-1})_*\left(B\left(\frac{\partial}{\partial t}\Sigma_s(t),\frac{\partial}{\partial s}\Sigma_s(t)\right)\right)\text{Hol}_{\Sigma_s}^{-1}dtds
    \end{equation}
    is equal the iterated integral
    \[
    \sum_{m,n\geq0}\int_{\Delta_{m+n+1}\times I}a_s(1-t_1)\cdots a_s(1-t_m)b_s(1-t_{m+1})a_s(1-t_{m+2})\cdots a_s(1-t_{m+n+1})dt_1\cdots dt_{m+n+1}ds.
    \]
  \end{lemma}
  \begin{proof}
    A solution to equation (\ref{poe}) in terms of iterated integrals is
    \[
    g_\gamma(t)=\text{id}+\sum_{n\geq1}\int_0^t\int_0^{t_1}\cdots\int_0^{t_{n-1}}a(t_1)\cdots a(t_n)dt_1\cdots dt_n.
    \]
    Similarly we have
    \[
    g_\gamma(t)^{-1}=\text{id}+\sum_{n\geq1}(-1)^n\int_0^t\int_0^{t_1}\cdots\int_0^{t_{n-1}}a(t_1)\cdots a(t_n)dt_1\cdots dt_n.
    \]
    and
    \[
    g_{\bar\gamma}(t)=\text{id}+\sum_{n\geq1}\int_0^t\int_0^{t_1}\cdots\int_0^{t_{n-1}}a(1-t_1)\cdots a(1-t_n)dt_1\cdots dt_n.
    \]
    replacing the previous expressions in (\ref{last}) yields the result.
  \end{proof}
\begin{lemma}\label{hol}
  The integral
  \[
    \sum_{m,n\geq0}\int_{\Delta_{m+n+1}\times I}a_s(1-t_1)\cdots a_s(1-t_m)b_s(1-t_{m+1})a_s(1-t_{m+2})\cdots a_s(1-t_{m+n+1})dt_1\cdots dt_{m+n+1}ds.
    \]
    is equal to $\mathsf{hol}(\sigma)$.
    \end{lemma}
\begin{proof}
The sum can be written as
    \begin{equation}\label{last2}
    \sum_{m,n\geq0}(-1)^{m+n}\int_{\Delta_{m+n+1}\times I}\mu^*_{(m+n+1)}(p_1^*\Sigma^*\omega_1\cdots p_m^*\Sigma^*\omega_1 p_{m+1}^*\Sigma^*(-\text{pr}(\omega_2)) p_{m+2}^*\Sigma^*\omega_1\cdots p_{m+n+1}^*\Sigma^*\omega_1),
    \end{equation}
    where $\mu_{(k)}:\Delta_k\times I\to (I^2)^{\times k}$ is defined by
    \[
    \mu_{(k)}(t_1,\cdots,t_k,s)=((1-t_1,s),\cdots,(1-t_k,s)).
    \]

Replacing $\Sigma=\sigma\circ\Theta_2$ in equation (\ref{last2}) we get
\[
\sum_{m,n\geq0}(-1)^{m+n+1}\int_{\Delta_{m+n+1}\times I}(\Theta_2)_{(m+n+1)}^*\left(p_1^*(\sigma^*(\omega_1))\wedge \cdots \wedge p_{m+1}^*(\sigma^*\omega_2)\wedge \cdots \wedge p_{m+n+1}^*(\sigma^*(\omega_1))\right),
\]
which, by degree reasons, is equal to
\[
\sum_{n\geq 0}(-1)^{n}\int_{\Delta_{n}\times I}(\Theta_2)_{(n)}^*\left(p_1^*(\omega)\wedge \cdots \wedge p_{n}^*(\omega_1)\right),
\]
the expression for $\beta_2(\sigma)$.
\end{proof}
\begin{TheoremA}
The map that associates a principal $2$-bundle with connection  to a higher local system is compatible with the construction of holonomies.  More precisely, the following diagram commutes:
  \begin{equation*}\label{comparison}
  \begin{CD}
    \Loc_\infty^{(E,\partial)}(M)@> K>> 2\text{-}\mathcal Bun^f_{\fr}(M)\\
    @V\mathcal{T}VV@VVTV\\
    \text{Rep}_{\leq2}(M,(E,\partial))@>> \mathcal{K}>\text{Rep}_{\leq2}(M,\fr)
  \end{CD}
\end{equation*}
The results above imply our main result
\end{TheoremA}
\begin{proof}
  We need to prove that two functors defined on the fundamental 2-groupoid of $M$ coincide. On simplices of dimension one, both functors are defined by
  ordinary holonomies, so they are equal. On the other hand, the fact that the two functors match on 2-morphisms is implied by Lemma \ref{hol}
  \end{proof}

\end{document}